\documentclass[12pt]{article}

\usepackage{amssymb}
\usepackage{amsfonts}
\usepackage{amsthm}
\usepackage{latexsym}
\usepackage[english]{babel}

\newcommand{\m}[1]{\marginpar{\tiny{#1}}}
\newcommand{\cA}{\mathcal A}
\newcommand{\cB}{\mathcal B}
\newcommand{\cC}{\mathcal C}
\newcommand{\cF}{\mathcal F}

\newcommand{\cT}{\mathcal T}

\newcommand{\pr}{{\mathbb P}}
\newcommand{\E}{{\mathbb E}}

\newcommand{\N}{\mathbb{N}}

\newcommand{\frag}{\rm frag}

\newtheorem{theorem}{Theorem}

\newtheorem{conjecture}[theorem]{Conjecture}

\begin{document}

\title{Connectivity for an unlabelled bridge-addable graph class}
\author{Colin McDiarmid\\ Oxford University}
%\date{13 March 2020}

\maketitle

Call a class $\cA$ of graphs \emph{bridge-addable} if, whenever a graph $G$ in $\cA$ has vertices $u$ and $v$ in distinct components, then the graph $G+uv$ (obtained by adding the edge $uv$) is also in $\cA$.
Thus for example the class $\cF$ of forests is bridge-addable, as is the class of planar graphs.

Let $\cA_n$ denote the set of graphs in $\cA$ on vertex set $\{1,\ldots,n\}$ (so these are labelled graphs), and similarly for other graph classes.
Also, let us use the notation $R_n \in_u \cA$ to indicate that the random graph $R_n$ is sampled uniformly from $\cA_n$ (and assume implicitly that $\cA_n$ is non-empty).  Thus
%\[ \pr(R_n \mbox{ is connected}) = \frac{|\cC_n|}{|\cA_n|}\]
\[ \pr(R_n \mbox{ is connected}) = \, |\cC_n|\, / \, |\cA_n|\]
  where $\cC$ is the class of connected graphs in~$\cA$.
 The following elementary but useful inequality appeared in McDiarmid, Steger and Welsh~\cite{msw05,msw06}.
  If $\cA$ is bridge-addable and $R_n \in_u \cA$ then %for each $n$ (assuming $\cA_n$ is non-empty)
\begin{equation} \label{eqn1}
  \pr(R_n \mbox{ is connected}) \geq e^{-1} \approx 0.3679 \;\; \mbox{ for each } n.
\end{equation}

If $\cT$ denotes the class of trees then $|\cT_n|/|\cF_n| \to e^{-\frac12}$ as $n \to \infty$, by a result of R\'enyi~\cite{renyi59} in 1959.
It was conjectured in~\cite{msw06} in 2006 that the class of forests is the `least connected' bridge-addable class; and in particular that,
if $\cA$ is bridge-addable and $R_n \in_u \cA$, then
\begin{equation} \label{conj1}
 \liminf_n   \pr(R_n \mbox{ is connected}) \geq e^{-\frac12}  \approx  0.6065 .
\end{equation}
After several partial results towards proving this conjecture, see~\cite{amr2012,bbg2008,bbg2010,kp2013,norin2013},
the full result was proved by Guillaume Chapuy and Guillem Perarnau~\cite{cp2019}.  See also~\cite{cmcd2013,cmcd2016,mw2014,mw2017}.
\medskip

Here we are interested in {\bf unlabelled} graphs, where by contrast little seems to be known.
Let $\tilde{\cA}_n$ be the set of unlabelled graphs in $\cA_n$, and similarly for other graph classes.  We may identify an unlabelled graph as an isomorphism class of labelled graphs.  We use the notation $\tilde{R}_n \in_u \tilde{\cA}$ to indicate that the random unlabelled graph $\tilde{R}_n$ is sampled uniformly from $\tilde{\cA}_n$.  Of course, as before
%\[ \pr(\tilde{R}_n \mbox{ is connected}) = \frac{|\tilde{\cC}_n|}{|\tilde{\cA}_n|}. \]
\[ \pr(\tilde{R}_n \mbox{ is connected}) = \, |\tilde{\cC}_n| \, / \,|\tilde{\cA}_n|. \]
We make two conjectures about the probability of being connected, corresponding to~(\ref{eqn1}) and~(\ref{conj1}) above. The first seems to ask for rather little (though see Theorem~\ref{thm.uconn} below).
%\medskip

%\noindent {\bf Conjecture 1} \,
\begin{conjecture} \label{conj.1}
  There is a $\delta>0$ such that, if the graph class $\cA$ is bridge-addable and $\tilde{R}_n \in_u \tilde{\cA}$, then
% for each $n$ (assuming $\cA_n$ is non-empty)
\begin{equation} \label{eqn3}
 \pr(\tilde{R}_n \mbox{ is connected}) \geq \delta \;\; \mbox{ for each } n.
\end{equation}
\end{conjecture}
For trees and forests we have
\[ |\tilde{\cT}_n|/|\tilde{\cF}_n| \to \tau \approx e^{-0.5226} \approx  0.5930 \;\; \mbox{ as } n \to \infty,\]
see for example the first line of table 3 in~\cite{km}.
The second conjecture 
%corresponds to~(\ref{conj1}) above, and
is more speculative.
Call the class $\cA$ of graphs \emph{decomposable} when a graph is in the class if and only if each component is.  Examples include the class $\cF$ of forests and the class of planar graphs.
%\medskip

%\noindent {\bf Conjecture 2} \,
\begin{conjecture} \label{conj.1}
  If the graph class $\cA$ is  bridge-addable and decomposable, and $\tilde{R}_n \in_u \tilde{\cA}$, then
\begin{equation} \label{eqn2}
 \liminf_n \pr(\tilde{R}_n \mbox{ is connected}) \geq \tau.
\end{equation}
\end{conjecture}

The {\em fragment size} $\mbox{frag}(G)$ of a graph $G$ is the number of vertices of $G$ less the maximum number of vertices in a component.
% (the number of vertices missing from a largest component).
In the labelled case, if the graph class $\cA$ is bridge-addable and $R_n \in_u \cA$ then (\cite{cmcd2013}, see~\cite{rgmc} for an earlier version)
\[ \E[ \mbox{frag}(R_n)] <2.\]
Is there a corresponding result in the unlabelled case?   Here is an awkward example.

Fix an integer $k \geq 3$, and let $\cA$ consist of all graphs $G$ such that deleting any bridges from $G$ yields a disjoint union of cycles each of length at least $k$.  Clearly $\cA$ is bridge-addable and decomposable.  Let $n=2k$, and let $\tilde{R}_n \in_u \tilde{\cA}$.  Then $\pr(\tilde{R}_n \mbox{ is connected})= \frac12$, which is satisfactory (though less than $\tau$); but $\E[\mbox{frag}(\tilde{R}_n)]=n/4$, which is bad!
%\smallskip

%\noindent {\bf Conjecture 3} \,
\begin{conjecture} \label{conj.3}
For each graph class $\cA$ which is bridge-addable and decomposable, there is a constant $c_{\cA}$ such that, for each positive integer $n$, if $\tilde{R}_n \in_u \tilde{\cA}$ then 
\[ \E[ \frag(\tilde{R}_n)] \leq c_{\cA}.\]
\end{conjecture}
\smallskip

The fourth and final conjecture is a little speculative (as was the second).

%\begin{conjecture} \label{conj.4} There is a constant $c$ such that, for each graph class $\cA$ which is bridge-addable and decomposable, if $\tilde{R}_n \in_u \tilde{\cA}$ then \[ \limsup_{n \to \infty} \, \E[ \frag(\tilde{R}_n)] \leq c.\]\end{conjecture}
%
\begin{conjecture} \label{conj.4}
There is a constant $c$ such that, for each graph class $\cA$ which is bridge-addable and decomposable,
% there exists $n_0=n_0(\cA)$ such that for $n \geq n_0$ and 
for $\tilde{R}_n \in_u \tilde{\cA}$ 
%\[ \E[ \frag(\tilde{R}_n)] \leq c.\]
\[ \limsup_{n \to \infty} \, \E[ \frag(\tilde{R}_n)] \leq c.\]
\end{conjecture}
\bigskip

Let us return to Conjecture~1, but lower our sights.  We can quickly prove a weakened version of the conjecture (with $\delta$ replaced by $1/2n$) which is still strong enough to be useful (see the proof of Theorem 4 in~\cite{ms2020}).

\begin{theorem} \label{thm.uconn}
Let $\cA$ be a bridge-addable class of graphs, let $\cC$ be the class of connected graphs in $\cA$, and let ${\cC}^{\bullet}$ be the set of vertex-rooted graphs in~${\cC}$. Then 
\begin{equation} \label{eqn.ubadd}
|\widetilde{\cA}_n | \leq 2\, |\widetilde{\cC}^{\bullet}_n | \leq  2n\, |\widetilde{\cC}_n |  \;\; \mbox{ for each } n \in \N \, .
\end{equation}
\end{theorem}

\begin{proof}
From each graph $G \in \widetilde{\cA}_n$ we shall construct a vertex-rooted connected graph $H^v \in \widetilde{\cC}^{\bullet}_n$ and $b \in \{0,1\}$, such that from $H^v$ and $b$ we can determine~$G$.
It will follow that $|\widetilde{\cA}_n | \leq 2\, |\widetilde{\cC}^{\bullet}_n |$, and since $|\widetilde{\cC}^{\bullet}_n| \leqslant n \, |\widetilde{\cC}_n|$ the proof will be complete.

Now for the construction.
Observe first that every $n$-vertex graph has at most $n-1$ bridges, so there is a vertex incident to at most one bridge.  
Let $G \in \widetilde{\cA}_n$. Let $C$ be a smallest component in $G$, and let $v$ be a vertex in $C$ incident to $b \in \{0, 1\}$ bridges.  Add an edge between $v$ and each component of $G$ other than $C$, forming the graph $H$.  Finally root $H$ at $v$, and %let $v$ be its root vertex.
output the rooted graph $H^v$ together with $b$.

Let $B$ be the set of bridges in $H$ incident to $v$. Then $B$ consists of $b$ edges from the original graph $G$ and any edges added when forming $H$.  
For each edge $e \in B$ let $n(e)$ be the number of vertices in the component of $H \backslash e$ not containing $v$. If $b=1$ and $e$ is the edge in $B$ from $G$ then $n(e) \leq v(C)-1$.  Also $n(e') \geq v(C)$ for any other edge $e' \in B$. Thus, given $H^v$ and $b$, we can identify these added edges and delete them to recover $G$, as required.
\end{proof}

\bigskip

%arXiv:2001.05256v1 (15 Jan 2020)

This note grew from an open problem presented at Oberwolfach in May 2014. 
\m{Barbados?}
 It was updated and presented as an open problem at the First Armenian Workshop on Graphs, Combinatorics, Probability, in Dzoraget, Armenia, in June 2019.  Now it has been amplified by adding Theorem~\ref{thm.uconn} (and some further references).
%and 3 refs

%five years without real progress for unlabelled graphs!

\end{document}